\documentclass[11pt,a4paper,reqno]{article}
\usepackage[latin1]{inputenc}
\usepackage[english]{babel}
\usepackage{xcolor}
\usepackage{a4wide,xspace}
\usepackage{amsmath, delarray,enumerate,calc,amsthm}
\usepackage{amssymb}
\usepackage{bbm}
\usepackage{txfonts}
\usepackage{paralist}
\usepackage{authblk}
\usepackage{stmaryrd}
\usepackage[mathscr]{euscript}
\usepackage{scalerel,stackengine}

\newtheorem{theorem}{Theorem}[section]
\newtheorem{lemma}[theorem]{Lemma}
\newtheorem{proposition}[theorem]{Proposition}

\newtheorem{definition}[theorem]{Definition}

\newtheorem{remark}[theorem]{Remark}

\numberwithin{equation}{section}

\newtheorem{corollary}[theorem]{Corollary}
\newtheorem{fact}[theorem]{Fact}

\newtheorem{theoremIprime}{Theorem}

\newtheorem{theoremIIprime}{Theorem}

\newcommand{\supp}{\operatorname{supp}}

\newcommand{\Bap}{\mathcal{B}\hspace*{-1pt}{\mathsf{ap}}}
\newcommand{\cA}{{\mathcal A}}

\newcommand{\cB}{{\mathcal B}}

\newcommand{\cH}{{\mathcal H}}

\newcommand{\cP}{{\mathcal P}}

\newcommand{\cM}{{\mathcal M}}

\newcommand{\LL}{{\mathscr L}}
\newcommand{\BB}{{\mathscr B}}

\newcommand{\R}{{\mathbbm R}}

\newcommand{\Z}{{\mathbbm Z}}

\newcommand{\SSS}{{\mathbbm S}}

\newcommand{\hx}{{\tilde{x}}}
\newcommand{\hy}{{\tilde{y}}}
\newcommand{\hX}{{\widetilde{X}}}
\newcommand{\hT}{\widetilde{T}}
\newcommand{\hXprime}{{\widetilde{X}'}}
\newcommand{\hTprime}{\widetilde{T}'}

\newcommand{\nuW}{\nu_{\scriptscriptstyle W}}
\newcommand{\nudW}{\nu_{\scriptscriptstyle d+W}}

\newcommand{\nuWinv}{\nu_{\scriptscriptstyle W_{inv}}}

\newcommand{\nuWG}{\nu_{\scriptscriptstyle W}^{\scriptscriptstyle G}}

\newcommand{\MW}{\cM_{\scriptscriptstyle W}}

\newcommand{\MG}{\cM^{\scriptscriptstyle G}}

\newcommand{\MWG}{\cM^{\scriptscriptstyle G}_{\scriptscriptstyle W}}

\newcommand{\MWGprime}{\cM^{\scriptscriptstyle G}_{\scriptscriptstyle W'}}

\newcommand{\piG}{\pi^{\scriptscriptstyle G}}
\newcommand{\piGhat}{\pi^{\scriptscriptstyle \widehat{G}}}
\newcommand{\piH}{\pi^{\scriptscriptstyle H}}

\newcommand{\piGH}{\pi^{\scriptscriptstyle G\times H}}
\newcommand{\pihX}{\pi^{\scriptscriptstyle \hX}}

\newcommand{\QM}{Q_{\scriptscriptstyle W}}

\newcommand{\QMinv}{Q_{\scriptscriptstyle W_{inv}}}
\newcommand{\QMG}{Q_{\scriptscriptstyle W}^{\scriptscriptstyle G}}
\newcommand{\QMdWG}{Q_{\scriptscriptstyle d+W}^{\scriptscriptstyle G}}

\newcommand{\QMbar}{Q_{\scriptscriptstyle \overline{W}}}
\newcommand{\QMGprime}{Q_{\scriptscriptstyle W'}^{\scriptscriptstyle G}}

\newcommand{\oneone}{1-1\,}
\newcommand{\dL}{\mathrm{dens}(\LL)}
\newcommand{\nuG}{\nu^{\scriptscriptstyle G}}

\newcommand{\nufG}[1]{\nu_{\scriptscriptstyle #1}^{\scriptscriptstyle G}}

\newcommand{\omegaG}{\omega^{\scriptscriptstyle G}}

\stackMath
\newcommand\reallywidecheck[1]{%
\savestack{\tmpbox}{\stretchto{%
  \scaleto{%
    \scalerel*[\widthof{\ensuremath{#1}}]{\kern-.6pt\bigwedge\kern-.6pt}%
    {\rule[-\textheight/2]{1ex}{\textheight}}%WIDTH-LIMITED BIG WEDGE
  }{\textheight}%
}{0.5ex}}%
\stackon[1pt]{#1}{\scalebox{-1}{\tmpbox}}%
}

\begin{document}

\title{Spectrum of weak model sets with Borel windows\footnote{Dedicated to Bob Moody on the occasion of his $80^{th}$ birthday}}

\author{Gerhard Keller$^{1}$, Christoph Richard$^{1}$ and Nicolae Strungaru$^{2}$}
\affil{1: Department Mathematik, Universit\"at Erlangen-N\"urnberg \\ \hspace{0.5cm} 2:
Department of Mathematical Sciences, MacEwan University}
\date{}

\maketitle

\abstract{Consider the extended hull of a weak model set together with its natural shift action. Equip the extended hull with the Mirsky measure, which is a certain natural pattern frequency measure. It is known that the extended hull is a measure-theoretic factor of some group rotation, which is called the underlying torus. Among other results, in the article \textit{Periods and factors of weak model sets} we showed that the extended hull is isomorphic to a factor group of the torus, where certain periods of the window of the weak model set have been factored out. This was proved for weak model sets having a compact window. In this note, we argue that the same results hold for arbitrary measurable and relatively compact windows. Our arguments crucially rely on Moody's work on uniform distribution in model sets. We also discuss implications for the diffraction of such weak model sets and discuss a new class of examples which are generic for the Mirsky measure.}

\section{The result}

Throughout this article, we will adopt the setting of \cite{KR2015, KR19}. For the statement of our result to be self-contained, we briefly recall the main notation. Fix locally compact second countable abelian groups $G,H$ with Haar measures $m_G,m_H$, and consider a co-compact lattice $\LL$ in $G\times H$, that projects injectively to $G$ and densely to $H$.  A \emph{window} is a measurable relatively compact set $W\subseteq H$.  By the so-called cut-and-project construction, these ingredients produce a weak model set. Let us describe this using point measures instead of sets.  Consider the compact quotient group $\hX=(G\times H)/\LL$, which is sometimes called the \emph{torus}.  (The torus is denoted by $\hat X$ in~\cite{KR19}. We changed notation from {\em hat} into {\em tilde} in order not to get into conflict with the group dual and the Fourier transform.) Fix $\hx=x+\LL\in\hX$. The cut step  yields the configuration $\nuW(\hx)=\sum_{y\in (x+\LL)\cap (G\times W)}\delta_y$, where $\delta_y$ puts a unit mass at $y\in G\times H$. The projection step maps the configuration $\nuW(\hx)$ to $G$, using the canonical projection $\piG:G\times H\to G$. This gives rise to a point measure $\nuWG(\hx)=(\piG_*\circ \nuW)(\hx)$, which has uniformly discrete support. The set $\Lambda_W(\hx)=\supp(\nuWG(\hx))$ is called a \emph{weak model set},
and we sometimes abbreviate $\Lambda_W=\supp(\nuWG(\tilde 0))$. The vague closure of $\nuWG(\hX)$ in the space of regular Borel measures on $G$ is called the \emph{extended hull} $\MWG$. The natural translation action $T$ on $G$, given by group addition $T_{g}g'=g+g'$, induces a translation action $\hT$ on $\hX$ by  $\hT_g \hx=(g,0)+\hx$ and an action $S$ on $\MWG$ by $(S_g\nu)(A)=\nu(T_g^{-1}A)$. Let us denote by $m_\hX$ the normalised Haar measure on $\hX$. Since $\nuWG$ is a measurable mapping, $(\MWG,S)$ carries a natural ergodic probability measure $\QMG=m_\hX\circ (\nuWG)^{-1}$, the so-called \emph{Mirsky measure}.

\smallskip

We have the following new results for the Mirsky measure  on the extended hull. These generalise Theorems B1 and B2 from \cite{KR19}, which were formulated for measurable, relatively compact windows $W\subseteq H$ that are  \emph{compact modulo $0$} \cite[Def.~3.5]{KR19}, i.e., there exist a compact set $K$ and set $N$ of zero Haar measure such that $W=K\triangle N$.
For the first result, recall that $W$ is \emph{Haar aperiodic} if $m_H((h+W)\triangle W)=0$ implies $h=0$. In Euclidean space, any nonempty window is Haar aperiodic.

\setcounter{theoremIprime}{1}

\begin{theoremIprime}\label{theo:B1gen}
Suppose that  $W$ is measurable, relatively compact and Haar aperiodic. Then $(\MWG,\QMG,S)$ is measure-theoretically isomorphic to $(\hX,m_\hX, \hT)$.
\end{theoremIprime}

For the general case, consider the group $H_W^{Haar}=\{h\in H: m_H((h+W)\triangle W)=0\} $ of Haar periods of $W$.
Write $\cH_W^{Haar}=\{0\}\times H_W^{Haar}$ for the canonical embedding of $H_W^{Haar}$ into $G\times H$.

\begin{theoremIIprime}\label{theo:B2gen}
Suppose that $W$ is measurable, relatively compact and $m_H(W)>0$. Let $\hXprime=\hX/\pihX(\cH_W^{Haar})$ with induced $G$-action $\hTprime$ and Haar measure $m_{\hXprime}$. Then $(\MWG,\QMG, S)$ is measure-theoretically isomorphic to $(\hXprime, m_\hXprime,\hTprime)$.
\end{theoremIIprime}

\begin{remark}[diffraction analysis]
The above result implies the known fact that the extended hull has pure point dynamical spectrum when equipped with the Mirsky measure, compare e.g.~\cite[Thm.~2a)]{KR2015}. In addition, the isomorphism in Theorem~\ref{theo:B2gen} explicitly describes the eigenvalues of the dynamical spectrum.  This is particularly useful for diffraction analysis as discussed in
Section~\ref{sec:Nicu-rel}, compare also the introduction to \cite{KR2015}. Let us mention here that $\hXprime$ characterises the group generated by the Bragg peak positions in the diffraction spectrum, i.e., that group is given by the $\widehat G$-projection of the group dual to $\hXprime$, which is viewed as a subgroup of $G\times H$. For details, see remarks~\ref{rem:ext} and~\ref{rem:dualgroup}.
\end{remark}

\begin{remark}[examples]
The above diffraction properties are realised by configurations which are generic for the Mirsky measure.
The precise connection is somewhat subtle, as Mirsky genericity on $G$ and on $G\times H$ have to be distinguished, see Theorem~\ref{thm.4.6} and Remark~\ref{rem:conclusions}.
For windows having almost no outer boundary, it is known that maximal density implies Mirsky genericity, see Remark~\ref{rem:cmg} below. Likewise, for windows having almost no inner boundary, minimal density implies Mirsky genericity.
Examples beyond these cases will be discussed in Section~\ref{sec:examples}.
\end{remark}

\section{Proof ingredients}

\subsection{Moody's uniform distribution theorem}

We will use a refinement of Moody's theorem on uniform distribution \cite[Thm.~1]{Moody2002}, which characterises sets of almost everywhere convergence. We first introduce the relevant notation. Consider any van Hove sequence $\cA=(A_n)_n$ in $G$ for averaging, see \cite[Eq.~(4)]{Moody2002} for a definition. Recall that $\nuG\in\MWG$ is Mirsky generic along $\cA$ if for every test function $\phi\in C(\MWG)$ the ergodic limit holds for the Mirsky measure $\QMG$ along $\cA$, i.e., we have
\begin{displaymath}
\lim_{n\to\infty} \frac{1}{m_G(A_n)} \int_{A_n} \phi(S_g \nuG) \, {\rm d}m_G(g)= \QMG(\phi) \ .
\end{displaymath}
In the sequel, we will consider ergodic limits on subclasses of test functions.
\begin{definition}[Mirsky $k$-genericity]\label{dfe:smg}
Let $\cA=(A_n)_n$ be any van Hove sequence in $G$, and let $k\in\mathbb N$.
We call $\nuG\in\MWG$ \textit{Mirsky $k$-generic} along $\cA$, if the ergodic limit holds for the Mirsky measure $\QMG$ along $\cA$, for every test function $\phi\in C(\MWG)$ of the form $\phi=\phi_{c_1}\cdot\ldots\cdot \phi_{c_k}\in C(\MWG)$, with $\phi_{c_i}\in C(\MWG)$ given by $\phi_{c_i}(\nu)=\nu(c_i)$ for $c_i\in C_c(G)$.
\end{definition}

Likewise, we will consider Mirsky genericity and Mirsky $k$-genericity of $\nu\in\MW$, i.e.,  with respect to the Mirsky measure $\QM=m_\hX\circ (\nuW)^{-1}$ along $\cA$, where $\MW$ denotes the vague closure of $\nuW(\hX)$ in the space $\cM$ of regular Borel measures on $G\times H$.

\begin{remark}[sets of Mirsky genericity]\label{rem:smg}
Fix  any \emph{tempered} van Hove sequence $\cA$ in $G$, see \cite[Eq.~(5)]{Moody2002} for a definition, and consider the set $\hX_{gen}=\hX_{gen}({\cA})$ of points $\hx\in\hX$ for which $\nuW(\hx)$ is Mirsky generic along $\cA$.  Note that $\hX_{gen}$  has full $m_\hX$-measure in $\hX$, which is seen as in the case of $\mathbb Z$-actions, see  e.g.~\cite[Cor.~4.20]{EW11}. Here we use that $G\times H$ is second countable and that $\MW$ is compact metrizable. This allows us to apply the Lindenstrauss ergodic theorem \cite[Thm.~1.2]{Lindenstrauss2001}, which holds for van Hove sequences that are tempered. For the existence of such averaging sequences, see e.g.~the discussion in \cite[Rem.~2.12 (v)]{mr13}. In particular, corresponding sets $\hX_k\supseteq \hX_{k+1}\supseteq \hX_{gen}$ for Mirsky $k$-genericity also have full $m_\hX$-measure, and we have $\hX_{gen}=\bigcap_{k\in \mathbb N} \hX_{k}$ by the Stone-Weierstrass theorem.
Observe that Mirsky $k$-genericity of $\nuWG(\hx)$ is inherited from Mirsky $k$-genericity of $\nuW(\hx)$, by continuity of the projection map $\piG_*$. Thus all sets of $k$-genericity for the Mirsky measure $\QMG$ are full $m_{\hX}$-measure sets. Moreover all of the above sets are $\hT$-invariant, as a consequence of the van Hove property.
\end{remark}

For the following proposition, note that for $\eta\in C_c(H)$ we have
\begin{displaymath}
(\eta\circ \piH)\cdot \nuW(\hx)=\sum_{y\in (x+\LL)\cap (G\times W)} \eta(y_H) \cdot\delta_y \ ,
\end{displaymath}
where we use the notation $y=(y_G,y_H)$ for $y\in G\times H$.

\begin{proposition}[Moody's uniform distribution theorem]\label{prop:Moodyext}
Assume that $W\subseteq H$ is relatively compact and measurable. Let $\cA=(A_n)_n$ be any van Hove sequence in $G$. Then the following hold.

\begin{itemize}
\item[(a)] The configuration $\nuWG(\hx)$ is Mirsky 1-generic along $-\cA$ if and only if
\begin{equation}\label{eq:maxdens}
\lim_{n\to\infty} \frac{\nuWG(\hx)(A_n)}{m_G(A_n)}=\lim_{n\to\infty} \frac{\nuW(\hx)(A_n\times H)}{m_G(A_n)}=\dL\cdot m_H(W) \ .
\end{equation}
\item[(b)] The configuration $\nuW(\hx)$ is Mirsky 1-generic along $-\cA$ if and only if
\begin{equation}\label{eq:Moody}
\lim_{n\to\infty} \frac{((\eta\circ \piH)\cdot \nuW(\hx))(A_n\times H)}{m_G(A_n)}=\dL\cdot m_H(\eta \cdot 1_W)
\end{equation}
for any $\eta\in C_c(H)$.
\end{itemize}
\end{proposition}

\begin{remark}[when Mirsky 1-genericity implies Mirsky genericity]\label{rem:cmg}
Consider any relatively compact and measurable window $W\subseteq H$. As limiting point frequencies of $\nuWG(\hx)$  always lie between $\dL\cdot m_H(W^\circ)$ and $\dL\cdot m_H(\overline{W})$, see e.g.~ \cite[Prop.~3.4]{HuckRichard15}, we say that $\nuWG(\hx)$ has maximal density along $\cA$ if the limit on the lhs in  Eq.~\eqref{eq:maxdens} equals $\dL\cdot m_H(\overline{W})$. As discussed in \cite[Rem.~3.16]{KR2015} and \cite[Rem.~8.7]{KR19}, maximal density of $\nuWG(\hx)$ along $\cA$ implies genericity of  $\nuW(\hx)$ along $-\cA$ with respect to the Mirsky measure $\QMbar$ on $G\times H$. One concludes that maximal density of $\nuWG(\hx)$ implies genericity of $\nuW(\hx)$ with respect to the Mirsky measure $\QM$ on $G\times H$ if and only if the  window satisfies $m_H(W)=m_H(\overline{W})$. This is a considerably stronger condition than the window being compact modulo $0$. Likewise, we speak of minimal density if the limit on the lhs in  Eq.~\eqref{eq:maxdens}  equals $\dL\cdot m_H(W^\circ)$. Minimal density of $\nuWG(\hx)$ along $\cA$ implies Mirsky genericity of $\nuW(\hx)$ along $-\cA$ if and only if $m_H(W)=m_H(W^\circ)$. See \cite[Thm.~17, Rem.~5]{BHS} for a variant of these results. An extension will be given in Lemma~\ref{lemma:generic} and Remark~\ref{rem:approx} below.
\end{remark}

\begin{remark}[Mirsky genericity on $G$ versus $G \times H$] Let us emphasize here that for each $W$ and $d \in H$ we have $\QMdWG=\QMG$, which follows
from the invariance of the Haar measure on $\tilde{X}$ under translation by $(0,d)+\LL$. Indeed, denoting $S_{(0,d)}:G\times H\to G\times H, x\mapsto x+(0,d)$ and
$\sigma_d:\MW\to\MW, \sigma_d\nu=\nu\circ S_{(0,d)}^{-1}$, we have
\begin{equation*}
\begin{split}
\QMdWG
&=
m_{\hX}\circ\nudW^{-1}\circ(\piG_*)^{-1}
=
m_{\hX}\circ(\sigma_d\circ\nuW\circ S_{(0,d)}^{-1})^{-1}\circ(\piG_*)^{-1}\\
&=
(m_{\hX}\circ S_{(0,d)})\circ \nuW^{-1} \circ(\piG_*\circ\sigma_d)^{-1}
=
m_{\hX}\circ\nuW^{-1}\circ(\piG_*)^{-1}=\QMG.
\end{split}
\end{equation*}
On another hand, if $\nuW(\hx)$ is Mirsky generic, then \eqref{eq:Moody} uniquely identifies the measure $\eta \mapsto m_H(\eta \cdot 1_W)$ for $\eta\in C_c(H)$, i.e., the Haar measure restricted to $W$. It follows immediately that $Q_{d+W}=Q_W$ if and only if $d$ is a Haar period for $W$.
\end{remark}

The above result can, with some adaptions, be proved as in \cite{Moody2002}. We start with the following lemma which slightly refines \cite[Prop.~2]{Moody2002}.

\begin{lemma}\label{lem:moodprep}
Let $W\subseteq H$ be relatively compact and nonempty. Then any point in $G$ has a compact neighborhood $B$ such that $((B-B)\times (W-W))\cap \LL=\{(0,0)\}$. As a consequence, for every $\hx\in \hX$ and every $g,g'\in \Lambda_W(\hx)$,\, $(g+B)\cap(g'+B)\ne\varnothing$ implies $g=g'$. The latter statement also holds with $W$ replaced by $-W$. Moreover the following are equivalent:
\begin{itemize}
\item[(i)] $((B-B)\times (W-W))\cap \LL=\{(0,0)\}$.
\item[(ii)] $\pi_{\hX}$ is one-to-one on $B\times W$.
\item[(iii)] $\Lambda_{W-W}\cap (B-B)=\{0\}$.
\end{itemize}
\end{lemma}

\begin{proof}
For the existence statement, take any compact zero neighborhood $U\subseteq G$ and note that $(U\times (W-W)) \cap \LL$ is finite as $\LL$ is locally finite. Hence there is a zero neighborhood $V\subseteq U$ such that $(V\times (W-W))\cap\LL=\{(0,0)\}$. The first claim follows after choosing a compact neighborhood $B$ of the given point in $G$ such that $B-B\subseteq V$.

\smallskip

For the second claim let $g,g'\in \Lambda_W(\hx)$ such that $(g+B)\cap(g'+B)\ne \varnothing$. Then there exist $h,h'\in H$ such that $(g,h), (g',h')\in (G\times W) \cap (x+\LL)$. As $(g+B)\cap(g'+B)\ne \varnothing$, this implies $(g-g',h-h') \in ((B-B)\times (W-W)) \cap  \LL$. Hence $g=g'$. Note that replacing $W$ by $-W$ does not alter the argument.

\smallskip

\noindent $(i) \Rightarrow (ii):$ Consider $(g,h), (g',h')\in B\times W$ such that $(g,h)=(g',h')+\ell$ for some $\ell\in \LL$. We then have $(g-g',h-h')\in ((B-B)\times (W-W))\cap \LL$. Hence $(g-g',h-h')=(0,0)$, and the claim follows.

\noindent $(ii) \Rightarrow (iii):$ Let $g\in \Lambda_{W-W}\cap (B-B)$. Then there exists $h\in H$ such that $(g,h)\in ((B-B)\times (W-W))\cap \LL$. Then $g=b-b'$ for some $b,b'\in B$ and $h=w-w'$ for some $w,w'\in W$, and $(b,w)=(b',w')+\ell$ for  $\ell=(g,h)\in \LL$. Hence $(b,w)=(b',w')$, which implies $g=0$.

\noindent $(iii) \Rightarrow (i):$ Assume $(g,h)\in ((B-B)\times (W-W))\cap \LL$. Then $g\in \Lambda_{W-W}\cap (B-B)$, which implies $g=0$. As $\piGH$ is  \oneone on $\LL$, this implies $h=0$, and the claim follows.

\end{proof}

\begin{proof}\textit{(Proof of Proposition~\ref{prop:Moodyext})}
We treat assertion (b) first. By a standard denseness argument, it suffices to consider functions $c\in C_c(G\times H)$ of product type $c=\psi\circ \piG \cdot \eta\circ \piH$ where $\psi\in C_c(G)$ and $\eta\in C_c(H)$.
Recalling $\phi_c(\nu)=\nu(c)$, we have for $\hy=y+\LL=(y_G,y_H)+\LL$ that
\begin{displaymath}
\phi_c(\nuW(\hy)) =\sum_{z\in (y+\LL)\cap (G\times W)}  \psi(z_G)\cdot \eta(z_H) =\sum_{\ell\in\LL}  c(y+\ell)\cdot 1_W(y_H+\ell_H)=
\tilde f(\hy)\ ,
\end{displaymath}
where  $\tilde f\in \mathcal L^1(\hX, m_{\hX})$ denotes the projected $\LL$-periodisation of the function $y\mapsto f(y)=c(y)\cdot 1_W(y_H)$.
Using the extended Weil formula \cite[Thm.~3.4.6]{RS}, we thus get
\begin{displaymath}
\begin{split}
\int_{\MW} \phi_c \, {\rm d} \QM
&=
\int_{\hX}\phi_c(\nuW(\hy))   \, {\rm d}m_{\hX}(\hy)
= \int_{\hX} \tilde f(\hy) \, {\rm d}m_{\hX}(\hy) = m_{\hX}(\tilde f)\\
&=\dL\cdot m_{G\times H}(f)=
\dL\cdot m_G(\psi)\cdot m_H( \eta\cdot 1_W) \ .
\end{split}
\end{displaymath}
Next, consider the $G$-orbit of any $\hx\in \hX$.
Here we assume without loss of generality that $\psi\in C_c(G)$ has sufficiently small support such that $B=\supp(\psi)$ satisfies the assumption in Lemma~\ref{lem:moodprep}. (The general case of arbitrary compact support can be treated using a partion of unity by functions of small support.)
Write $Y=B\times W\subseteq G\times H$ and define $\widetilde Y=\pi_{\hX}(B\times W)\subseteq \hX$. It is readily seen that $\hT_g\hx \in \widetilde Y$ if and only if $g\in -\Lambda_W(\hx)+B$. In particular, in that case there exists $\ell\in \LL$ such that $(x+\ell)_G\in \Lambda_W(\hx)$, $g\in -(x+\ell)_G+B$, and $\tilde f(\widetilde T_g\hx)=\psi((x+\ell)_G)\cdot\eta((x+\ell)_H)\cdot 1_W((x+\ell)_H)$ as $\pi_{\hX}$ is 1-1 on $Y$. Note that $-\Lambda_W(\hx)+B$
is a pairwise disjoint union of translates of $B$, which follows from Lemma~\ref{lem:moodprep} as $-\Lambda_W(\hx)=\Lambda_{-W}(-\hx)$. As  $S_g \nuW(\hx)=\nuW(\hT_g\hx)$, we thus have by the van Hove property of $(-A_n)_n$ that
\begin{displaymath}
\begin{split}
\lim_{n\to\infty} \frac{1}{m_G(A_n)}& \int_{-A_n} \phi_c(S_g \nuW(\hx)) \, {\rm d}m_G(g)
=\lim_{n\to\infty} \frac{1}{m_G(A_n)} \int_{-A_n} \tilde f(\hT_g\hx) \, {\rm d}m_G(g)
\\
&=\lim_{n\to\infty} \frac{1}{m_G(A_n)} \sum_{y_G\in \Lambda_W(\hx)\cap A_n} \eta(y_H) \cdot m_G(\psi)\\
&=m_G(\psi)\cdot \lim_{n\to\infty} \frac{((\eta\circ \piH)\cdot \nuW(\hx))(A_n\times H)}{m_G(A_n)}\ ,
\end{split}
\end{displaymath}
provided that the above limit exists. Now the claim in part (b) is obvious.
\smallskip

\noindent The proof of (a) is analogous: reread the above proof of (b) for $\eta\equiv 1$, considering functions $c\in C_c(G)$ and 1-genericity with respect to $\QMG$.
\end{proof}

\subsection{Haar periods and periods}

For the following recall the notion of period group $H_W=\{h\in H: h+W=W\}$ and of Haar period group $H_W^{Haar}=\{h \in H: m_H((h+W)\Delta W)=0\}$. Then $W$ is called (Haar) aperiodic if its (Haar) period group is trivial.
In order to apply the techniques from \cite{KR19} with only minimal changes, we will circumvent the notion of  Haar regularity \cite[Rem.~3.12]{KR19}, which relies on compactness. Instead we will construct a measurable version $W_{inv}$ of $W$, which coincides with $W$ up to measure zero, but is strictly invariant under translation by any $h\in H_W^{Haar}$.

\smallskip

We start by reviewing some simple properties of $H_W^{Haar}$, which are listed in \cite[Lem.~7.1]{S20}, see also \cite[Fact~2]{BHS}. For completeness we include the straightforward proofs.
First, let us recall that for a measurable relatively compact set $W \subseteq H$, its covariogram function $c_W$ is defined via
\begin{equation}\label{eq:cov}
c_W:=1_W*1_{-W} \ ,
\end{equation}
where $*$ denotes convolution.
Note that $c_W$ is a positive definite function, which is continuous by \cite[Thm.~I.1.6 (b)]{Rud}, \cite[Prop.~3.6.3]{RS} and obviously has compact support. A simple computation yields for any $h\in H$ the relation
\begin{equation}\label{eq1}
m_H((h+W)\triangle W)=2\cdot m_H(W\setminus (W+h))=2\cdot \left(c_W(0)- c_W(h)\right) \ .
\end{equation}
We have the following characterisation of $H_W^{Haar}$.

\begin{lemma}\cite[Lem.~7.1]{S20}
Assume that $W\subseteq H$ is relatively compact and measurable. Then
\begin{displaymath}
\begin{split}
 H_W^{Haar}  &= \{ h \in H : \|1_W-T_h1_W \|_1 =0 \} = \{h \in H : c_W(h)=c_W(0) \} \\
 & = \{ h \in H : T_hc_W= c_W\} \ ,
 \end{split}
\end{displaymath}
where $(T_hf)(y)=f(y-h)$ denotes translation in $H$.  In particular $ H_W^{Haar}$ is a compact group.
\end{lemma}

\begin{proof}
The first equality follows immediately from the observation $\|1_W-T_h1_W \|_1 = m_H(W \Delta (h+W))$,  while the second one follows from Eq.~\eqref{eq1}. For the last equality, the inclusion $\supseteq$ is obvious, while $\subseteq$ is an immediate consequence of Krein's inequality $|f(y-h)-f(y)|^2\le 2f(0)(f(0)-\mathrm{Re} f(h))$ for positive definite functions $f$, see e.g.~ \cite[Ch.~I.3.4]{BF}. Finally, since $c_W$ is a continuous function of compact support, its period group is closed and relatively compact, hence compact.
\end{proof}

We can now prove the existence of the measurable version $W_{inv}$ of $W$.

\begin{lemma}\label{lemma:W_inv}
There exists a measurable set $W_{inv}\subseteq H$ such that
\vspace{-1ex}
\begin{itemize}
\item[(a)] $m_H(W\triangle W_{inv})=0$ and
\item[(b)] $W_{inv}+h=W_{inv}$ for all $h\in H_W^{Haar}$.
\end{itemize}
\end{lemma}

\begin{proof}
Abbreviate $H_0:=H_W^{Haar}$ and denote by $m_{H_0}$ the normalized Haar measure on the compact abelian group $H_0$. Define
$\psi:H\to\R$ as the $H_0$-periodisation of $1_W$, i.e.,
$$
\psi(h):=\int 1_W(h+h_0)\,{\rm d}m_{H_0}(h_0) \ ,
$$
and let $W_{inv}:=\{h\in H:\psi(h)=1\}$.
As $m_{H_0}$ is translation invariant, we have $\psi(h+h_0)=\psi(h)$ for all $h_0\in H_0$, and assertion (b) follows at once.

We turn to assertion (a).
For measurable $A\subseteq H$ with $m_H(A)<\infty$ and all $h_0\in H_0$ we have:
\begin{equation*}
m_H(A\cap W)= m_H(A\cap (W-h_0))=\int_A1_W(h+h_0)\,{\rm d}m_H(h) \ .
\end{equation*}
Hence, using Fubini,
\begin{equation*}
\begin{split}
m_H(A\cap W)
&= \int\left(\int_A1_W(h+h_0)\,{\rm d}m_H(h)\right){\rm d}m_{H_0}(h_0)\\
&=
\int_A\left(\int1_W(h+h_0)\,{\rm d}m_{H_0}(h_0)\right){\rm d}m_H(h)
= \int_A\psi\,{\rm d}m_H \ .
\end{split}
\end{equation*}
As $m_H$ is $\sigma$-finite and this holds for all $A\subseteq H$ of finite measure, it follows that $1_W\cdot m_H=\psi \cdot m_H$, i.e., $1_W=\psi$ on a measurable set $H_1\subseteq H$ with $m_H(H\setminus H_1)=0$.
It follows that $W\cap H_1=W_{inv}\cap H_1$.
\end{proof}

The lemma has the following immediate corollary.
\begin{corollary}(Periods and Haar periods)\label{coro:Haar-and-other-periods}\\
We have $H_W^{Haar}=H_{W_{inv}}^{Haar}=H_{W_{inv}}$. In particular $W$ is Haar aperiodic if and only if $W_{inv}$ is aperiodic. \qed
\end{corollary}

Let $H':=H/H_{W_{inv}}$ and denote by $\varphi:H\to H'$ the canonical projection. Consider $W':=\varphi(W_{inv})$ and note $\varphi^{-1}(W')=W_{inv}$.

\begin{lemma}\label{lemma:W'}
$W'\subseteq H'$ is Borel measurable and Haar aperiodic in $H'$.
\end{lemma}

\begin{proof}
We first show measurability of $W'$. Let $W'':=\varphi(H\setminus W_{inv})$. Then $W'\cup W''=\varphi(H)=H'$, where $W'\cap W''=\varnothing$. Indeed, otherwise there are $h_1\in W_{inv}$ and $h_2\in H\setminus W_{inv}$ such that $\varphi(h_1)=\varphi(h_2)$. Then $h_2-h_1\in H_{W_{inv}}$, so that $W_{inv}+(h_2-h_1)=W_{inv}$. In particular $h_2=h_1+(h_2-h_1)\in W_{inv}$, a contradiction. As $W'$ and $W''=H'\setminus W'$ are both analytic sets \cite[(14.4)ii)]{Kechris}, they are Borel sets in view of Souslin's theorem \cite[(14.11)]{Kechris}. In order to show Haar aperiodicity, suppose that $m_{H'}((W'+h')\triangle W')=0$ for some $h'=\varphi(h)\in H'$, where $m_{H'}=m_H\circ\varphi^{-1}$. Then $0=m_H((W_{inv}+h+H_{W_{inv}})\triangle W_{inv})=m_H((W_{inv}+h)\triangle W_{inv})$,
so that $h\in H_{W_{inv}}^{Haar}=H_{W_{inv}}$, see Corollary~\ref{coro:Haar-and-other-periods}. Hence $h'=\varphi(h)$ is the neutral element in $H'$.
\end{proof}

\section{Proofs}

\subsection{Haar aperiodic windows}

Our proof of Theorem~\ref{theo:B1gen} uses Mirsky 1-generic configurations on $G\times H$ along some fixed tempered van Hove sequence.
Recall that the set $\hX_{1}\subseteq \hX$ from Remark~\ref{rem:smg} has full $m_\hX$-measure and is $\hT$-invariant.

\begin{lemma}\label{lem:Haarper}
Take $\hx,\hy\in \hX_{1}$ such that $\nuWG(\hx)=\nuWG(\hy)$. Then $\nuW(\hy)=\sigma_d\nuW(\hx)$ for some $d\in H$, where $(\sigma_d\nu)(A)=\nu(A-(0,d))$ for all Borel subsets $A$ of $G\times H$. Moreover $d$ is a Haar period of $W$.
\end{lemma}

\begin{proof}
Proposition~\ref{prop:Moodyext} (b) shows that for each $\hx\in\hX_{1}$ the sequence of measures $(\mu_n(\hx))_n$, defined by
\begin{equation}\label{eq:mun}
\mu_n(\hx)(\eta):= \frac{1}{\dL}
\frac{((\eta\circ \piH)\cdot \nuW(\hx))(A_n\times H)}{m_G(A_n)}
\end{equation}
for $\eta\in C_c(H)$, converges weakly to $m_H|_W$.
Take $\hx,\hy\in \hX_{1}$ such that $\nuWG(\hx)=\nuWG(\hy)$.
Then, by \cite[Lem.~4.4]{KR19}, there is $d\in H$ such that $\nuW(\hy)=\sigma_d\nuW(\hx)$.
As both sequences $(\mu_n(\hx))_n$ and $(\mu_n(\hy))_n$ converge weakly to $m_H|_W$ and as the translation $\sigma_d$ is weakly continuous, this shows that $\sigma_d(m_H|_W)=m_H|_W$, in particular $m_H((W-d)\cap W)=m_H(W)$. As $m_H(W)=m_H(W-d)$, this proves  $m_H((W-d)\triangle W)=0$, i.e., $d$ is a Haar period of~$W$.
\end{proof}

\begin{lemma}\label{lem:Hap}
Define $\MW'\subseteq \MW$ by $\MW'=\nuW(\hX_{1})$. If $W$ is Haar aperiodic, then $\piG_*|_{\MW'}: \MW'\to \MWG$ is \oneone.
\end{lemma}

\begin{proof}
Take $\hx,\hy\in \hX_{1}$ such that $\nuWG(\hx)=\nuWG(\hy)$.
Then by Lemma~\ref{lem:Haarper} we have $\nuW(\hy)=\sigma_d\nuW(\hx)$ for some Haar period $d$ of $W$. As $W$ is Haar aperiodic, we get $d=0$, that is $\nuW(\hx)=\nuW(\hy)$.
\end{proof}

\begin{proof}\textit{(Proof of Theorem \ref{theo:B1gen})}
$\piG_*$ is \oneone at $\QM$-a.a. $\nu\in\MW$ by Lemma~\ref{lem:Hap} and the fact that  $\hX_{1}$ has full $m_\hX$-measure by Remark \ref{rem:smg}. As $\hX_{1}$ is $\hT$-invariant, we conclude that $\piG_*:(\MW,\QM,S)\to(\MWG,\QMG,S)$ is a measure-theoretic isomorphism (observe the Lusin-Souslin theorem \cite[Thm.~15.1]{Kechris}). Also note that $\nuWG:(\hX,m_\hX, \hT)\to(\MW,\QM,S)$ is a measure-theoretic isomorphism by \cite[Thm.~2a)]{KR2015}. Here we use $m_H(W)>0$, which follows from Haar aperiodicity of $W$. Hence the claim is shown.
\end{proof}

\subsection{General windows}

Our proof of Theorem~\ref{theo:B2gen} proceeds by reduction to the Haar aperiodic case.
The construction of factoring out topological or measure-theoretic periods has been described in detail in Chapter 6 and 7 of \cite{KR19} for compact windows.  The same constructions can be used in the non-compact case. Since the group of Haar periods is closed, the quotient $\hXprime=\hX/\pihX(\cH_W^{Haar})$ is a compact abelian group.

\begin{proof}\textit{(Proof of Theorem~\ref{theo:B2gen})}
Assume first that $W=W_{inv}$. The set $W'=\varphi(W)$ is Haar aperiodic, see
Lemma~\ref{lemma:W'}.
Also note $(\MWG,\QMG,S)=(\MWGprime,\QMGprime,S)$, which follows with the same proof as in Proposition 6.10 in \cite{KR19}. Now the claim of the theorem follows from Theorem~\ref{theo:B1gen}.
In the general case, note that $(\MW,\QM,S)$ is measure-theoretically isomorphic to $(\cM,\QM,S)$. As the present theorem applies to the regularized window $W_{inv}$, it suffices to show that $\QM=m_\hX\circ (\nuW)^{-1}$ equals $\QMinv=m_\hX\circ (\nuWinv)^{-1}$ on $\cM$. But this follows from the observation
\begin{equation*}
\left\{\hx\in\hX:\nuW(\hx)\neq\nuWinv(\hx)\right\}
\subseteq
\pihX\left(\bigcup_{\ell\in\LL}\left((G\times(W\setminus W_{inv}))-\ell\right)\right)\ ,
\end{equation*}
and this is a set of $m_\hX$-measure zero, because $\LL$ is countable and $m_H(W\setminus W_{inv})=0$ by Lemma~\ref{lemma:W_inv}a).
\end{proof}

\section{Consequences for diffraction}\label{sec:Nicu-rel}

We discuss implications of our results for diffraction analysis of configurations, compare \cite{BaakeLenz2004, Lenz2009}. In particular, we discuss Besicovitch almost periodicity \cite{LSS2020a}, which links our approach to that in \cite{S20}. Whereas in the latter reference the Mirsky measure is constructed using Besicovitch almost periodic configurations, compare \cite[Thm.~6.13]{LSS2020a}, we take the Mirsky measure for granted and investigate when projections of Mirsky generic configurations are Besicovitch almost periodic. We assume that the reader is familiar with Remark 8.8 in \cite{KR19}, where the notions of autocorrelation measure, diffraction measure, diffraction spectrum and generic configuration are discussed in the present framework.

\smallskip

The link between dynamical and diffraction properties is well understood, see for example \cite[Sec.~6-8]{BaakeLenz2004}.  Let us specialise this to our needs.

\begin{fact}[\cite{BaakeLenz2004}, Sec.~6-8] \label{fact:1}
$(\MWG, \QMG, S)$ has  discrete dynamical $L^2$-spectrum. It also has pure point diffraction spectrum, i.e., its autocorrelation measure $\gamma_{\QMG}$, which is characterised via $\gamma_{\QMG}(c_1*c_2)=\QMG(\phi_{c_1}\cdot \phi_{c_2})$ for $c_1,c_2\in C_c(G)$, has a Fourier transform $\widehat{\gamma_{\QMG}}$ that is a point measure. The group $\SSS\subseteq \widehat G$ of dynamical eigenvalues of $(\MWG, \QMG, S)$ is generated by the set of Bragg peak positions, i.e., by those characters $\chi\in\widehat{G}$ for which $\widehat{\gamma_{\QMG}}(\{\chi\})\neq0$.
\end{fact}

\noindent Indeed, as $(\MWG, \QMG, S)$ is a factor of the system $(\hX,m_\hX, \hT)$ with factor map $\nuWG$, see \cite[Thm.~2]{KR2015}, and as the latter system has discrete dynamical spectrum, the same is true for $(\MWG, \QMG, S)$, and pure point diffraction spectrum as well as the remaining assertions  follow from  \cite[Thms.~7, \, 9]{BaakeLenz2004} and the Dworkin type calculation in the proof of Theorem 5 (a)  from \cite{BaakeLenz2004}.

\smallskip

In \cite{Lenz2009}, this link is analysed in more detail. Consider an eigenvalue $\chi\in \SSS$ and denote by $E_\chi$ the projection to the subspace of $L^2(\MWG, \QMG)$ generated by an eigenvector having eigenvalue $\chi$. If $c_\chi(\nuG):=(E_\chi \phi_{\overline{\chi}\cdot \sigma})(\nuG)$ does not vanish almost surely,  it gives a corresponding measurable eigenfunction, compare the proof of Theorem 3 in \cite{Lenz2009}. Here $\sigma\in C_c(G)$ is any function satisfying $m_G(\sigma)=1$.  Let us define $E_\chi=0$ if $\chi\notin\SSS$. Then, for any $\chi\in \widehat G$, the function $|c_\chi|$ is $\QMG$-almost surely constant by ergodicity.

\smallskip

Next, consider any van Hove sequence $\mathcal A=(A_n)_{n}$ in $G$.
For an individual configuration $\nuG\in \MWG$, the point part in its diffraction is often inferred from the so-called Fourier-Bohr coefficients along $\cA$, which are for $\chi\in \widehat G$ defined by
\begin{equation}\label{eq:FBlim}
a_\chi^{\mathcal A}(\nuG)= \lim_{n\to \infty} \frac{1}{m_G(A_n)} \int_{A_n} \overline{\chi(t)} \, {\rm d}\nuG(t) \ ,
\end{equation}
whenever that limit exists. We have the following result.

\begin{fact}[\cite{Lenz2009}, Thms.~3 and~5]\label{fact:L09}
Consider any $\chi\in \widehat G$. We then have $\widehat{\gamma_{\QMG}}(\{\chi\})=\langle c_\chi, c_\chi\rangle$, where $\langle \cdot , \cdot\rangle$ denotes the scalar product on $L^2(\MWG, \QMG)$.
Moreover, for any tempered van Hove sequence $\mathcal A=(A_n)_{n}$, the limit $a_\chi^{\mathcal A}$ along $\cA$ in Eq.~\eqref{eq:FBlim} exists in $L^2(\MWG,\QMG)$. In fact $a_\chi^{\mathcal A}=c_\chi$ holds $\QMG$-almost surely.  \qed
\end{fact}

\begin{remark}\label{rem:ext}
An eigenvalue $\chi\in \SSS$ may satisfy  $\widehat{\gamma_{\QMG}}(\{\chi\})=0$, in which case $\chi$ is called an \textit{extinction position}. Extinction positions have been observed for the Fibonacci chain, see e.g.~\cite[Sec.~9.4.1]{BG2}, where they reflect an inflation symmetry of the underlying point set.
Note that by Fact~\ref{fact:L09}, if $\chi\in \SSS$ is not an extinction position, i.e., if $\chi\in \SSS$ is a Bragg peak position, then $\nuG \mapsto a_\chi^{\mathcal A}(\nuG)$ defines the eigenfunction $c_\chi$ for $\chi$. On the other hand, if $\chi \in \SSS$ is an extinction position, then by Fact~\ref{fact:1}, there exist $\chi_1,\ldots ,\chi_{k}, \chi_{k+1}, \ldots , \chi_n \in \SSS$ which are Bragg peak positions so that $\chi= \chi_1 \cdot\ldots \cdot\chi_{k} \cdot \chi_{k+1}^{-1} \cdot \ldots \cdot \chi_n^{-1}$. In this case, an eigenfunction ${\widetilde c}_\chi$ is given by ${\widetilde c}_\chi =c_{\chi_1}\cdot \ldots \cdot c_{\chi_k} \cdot \overline{c_{\chi_{k+1}}}\cdot\ldots \cdot \overline{c_{\chi_{n}}}
$.
\end{remark}

\begin{remark}\label{rem:dualgroup}
Note that Theorem~\ref{theo:B2gen} explicitly describes the group of eigenvalues of $(\MWG,\QMG,S)$,  compare \cite[Ch.~7]{Lenz2009}. To explain this, denote by $\LL^\circ\subseteq \widehat G\times \widehat H$ the annihilator of the lattice $\LL\subseteq G\times H$, which is isomorphic to the group dual to $\hX$. Further, denote by ${\LL^\circ}'\subseteq \LL^\circ$ those characters whose $\widehat H$-component is $H_W^{Haar}$-invariant, i.e., we have ${\LL^\circ}'=\LL^\circ \cap (\widehat G\times (H_W^{Haar})^\circ)$, where $(H_W^{Haar})^\circ\subseteq \widehat H$ is the annihilator of $H_W^{Haar}$. Note that ${\LL^\circ}'$ is isomorphic to the group dual to $\hXprime$. The same statement holds for its projection $\piGhat({\LL^\circ}')$, as $\LL^\circ$ projects injectively to $\widehat G$. Now Theorem~\ref{theo:B2gen} implies that the group of eigenvalues is $\piGhat({\LL^\circ}')$. Theorem~\ref{theo:B2gen} also provides a way to compute the eigenfunctions via the torus parametrisation map, compare the above discussion.
\end{remark}

Although the notion of Besicovitch almost periodicity for a measure is known for quite some time \cite{Lag00, G05, BM}, it has only recently systematically been studied, in conjunction with other types of almost periodicity \cite{LSS2020a}. Fix any van Hove sequence  $\mathcal A=(A_n)_{n}$ in $G$ and consider the seminorm $\|\cdot\|_{2,\mathcal A}$, which is   for any $f\in L^2_{loc}(G)\cap L^\infty(G)$ defined by
\begin{displaymath}
\|f\|_{2,\mathcal A} = \limsup_{n\to\infty} \left(\frac{1}{m_G(A_n)} \int_{A_n} |f(t)|^2 \, {\rm d}t\right)^{1/2} .
\end{displaymath}
Then $f$ is called \emph{ Besicovitch almost periodic along $\cA$} if $f$ can be approximated by trigonometric polynomials with respect to $\|\cdot\|_{2,\mathcal A}$, see Definition 3.1 and Proposition 3.7 in \cite{LSS2020a}. A translation bounded measure $\nuG\in \MG$ is called \emph{Besicovitch almost periodic along $\mathcal{A}$} if the function $\varphi*\nuG$ is Besicovitch almost periodic for any $\varphi\in C_c(G)$, see Definition~3.30 and Remark 3.31 in \cite{LSS2020a}. The space of Besicovitch almost periodic measures is denoted by $\Bap_{\mathcal A}(G)$. This space is important in mathematical diffraction theory as it characterises pure point diffractive measures in the following sense \cite[Thm.~3.36]{LSS2020a}. We recall the definition of the autocorrelation $\gamma_{\nuG}$ of $\nuG$ along $\cA$,
\begin{equation}\label{eq:ac}
\gamma_{\nuG}
%:=\nuG \circledast_\cA \widetilde{\nuG} 
:= \lim_{n\to\infty} \frac{1}{m_G(A_n)}
\nuG|_{A_n} * \widetilde{\nuG|_{A_n}} \ ,
\end{equation}
whenever that limit exists. Here measure reflection is defined by $\widetilde\mu(f)=\overline{\mu(\widetilde{f})}$, where $\widetilde{f}(x)=\overline{f(-x)}$.

\begin{fact}[cf. Theorem 3.36 in \cite{LSS2020a}] \label{fact:bap}
Fix any van Hove sequence $\mathcal A=(A_n)_{n}$ in $G$ and let $\nuG\in \MG$ be a translation bounded measure.  Then $\nuG\in \Bap_{\mathcal A}(G)$ if and only if the following properties hold.
\begin{itemize}
\item[(i)] $\nuG$ has autocorrelation  $\gamma_{\nuG}$ along $\cA$, and $\widehat{\gamma_{\nuG}}$ is a pure point measure.
\item[(ii)] The Fourier-Bohr coefficients $a_\chi^\cA(\nuG)$ along $\cA$ exist for all $\chi\in \widehat G$.
\item[(iii)] The consistent phase property $\widehat{\gamma_{\nuG}}(\{ \chi \})=|a_\chi^\cA(\nuG)|^2$ holds for all $\chi\in \widehat G$.
\end{itemize} \qed
\end{fact}
We can now prove a strengthened version of Theorem 4.1 from \cite{S20}.
In order to simplify notation, we denote weighted model combs by
\begin{displaymath}
\nufG{h}(\hx)= \sum_{y\in (x+\LL)\cap (G\times H)}h(y_H)\cdot \delta_{y_G} \ .
\end{displaymath}
In particular we have $\nuWG(\hx)=\nufG{1_W}(\hx)$.

\begin{theorem}[cf. Theorem 4.1 from \cite{S20}]\label{thm.4.6}
Let $W \subseteq H$ be a relatively compact measurable window in some cut-and-project scheme $(G,H, \LL)$, where both $G$ and $H$ are second countable.  Let $\mathcal{A}=(A_n)_n$ be any van Hove sequence in $G$. Then the following hold.
\begin{enumerate}
  \item[(a)] $\nuWG(\hx)$ is Mirsky 1-generic along $-\cA$ if and only if $\nuWG(\hx)$ satisfies uniform distribution along $\cA$, i.e., if we have
  \[
  \lim_{n\to\infty} \frac{1}{m_G(A_n)} \nuWG(\hx)(A_n) = \dL \cdot m_H(W) \,.
  \]
\item[(b)] $\nuWG(\hx)$ is Mirsky 2-generic along $-\cA$ if and only if $\nuWG(\hx)$ has an autocorrelation $\gamma_{\nuWG(\hx)}$ along $\cA$ of the form
\[
\gamma_{\nuWG(\hx)}=\dL \cdot \nufG{c_W}(\tilde 0)\ ,
\]
where $c_W\in C_c(H)$ is the covariogram function of Eq.~\eqref{eq:cov}.
\item[(c)]
If $\nuWG(\hx)$ is Mirsky 2-generic along $-\cA$, then the Fourier transform of $\gamma_{\nuWG(\hx)}$  is given by
\[
\widehat{\gamma_{\nuWG(\hx)}}= \dL^2 \cdot  \sum_{\chi \in \piGhat(\LL^\circ)} \reallywidecheck{c_W}(\eta)  \cdot \delta_\chi \,,
\]
where $\eta\in \widehat H$ is uniquely determined by $(\chi,\eta)\in \LL^\circ$, and $\reallywidecheck{c_W}(\eta)=0$
if $\chi\in  \piGhat(\LL^\circ\setminus {\LL^\circ}')$, compare Remark~\ref{rem:dualgroup} for notation.
Observe also that $\reallywidecheck{c_W}(\eta)=|\reallywidecheck{1_W}(\eta)|^2$.

 \item[(d)] $\nuW(\hx)$ is Mirsky 1-generic  along $-\cA$ if and only if, for any $\chi \in \piGhat(\LL^\circ)$, the Fourier-Bohr coefficient of $\nuWG(\hx)$ along $\cA$ exists and is given by
\begin{equation}\label{eq:FB-coeff}
a_\chi^\cA(\nuWG(\hx))=
\dL \cdot \overline{\chi(x_G)} \cdot  \overline{\eta(x_H)} \cdot \reallywidecheck{1_W}(\eta) \, .
\end{equation}
Here $\hx=(x_G,x_H)+\LL$, and $\eta\in \widehat H$ is uniquely determined by $(\chi,\eta)\in \LL^\circ$, compare Remark~\ref{rem:dualgroup} for notation.

\item[(e)] Assume that $\nuWG(\hx)$ is Mirsky 2-generic along $-\cA$ and $\nuW(\hx)$ is Mirsky 1-generic  along $-\cA$. Then  $\nuWG(\hx)$ is Besicovitch almost periodic along $\cA$ if and only if $a_\chi^\cA(\nuWG(\hx))=0$ for all $\chi \in \widehat{G} \backslash \piGhat({\LL^\circ})$.

\end{enumerate}
\end{theorem}

\begin{remark}[relation to dynamical diffraction]\label{rem:conclusions} The proof of Theorem~\ref{thm.4.6} (b) shows that the autocorrelation $\gamma_{\nuWG(\hx)}$ agrees with the autocorrelation $\gamma_{\QMG}$ of $(\MWG, \QMG, S)$ from Fact~\ref{fact:1} if and only if $\nuWG(\hx)$ is Mirsky 2-generic along $-\cA$.

In dynamical diffraction analysis, people often consider the hull $\overline{\{S_g\nuG:g\in G\}}$ associated to a configuration $\nuG$.  For any configuration $\nuWG(\hx)$ that is Mirsky generic along $-\cA$, its hull $\MWG(\hx)=\overline{\{S_g\nuWG(\hx): g\in G\}}$ has full Mirsky measure $\QMG(\MWG(\hx))=1$. This is seen as in the proof of \cite[Thm.~5c)]{KR2015}. Thus in that case, the systems $(\MWG(\hx), \QMG,S)$ and $(\MWG, \QMG,S)$ are measure-theoretically isomorphic.
\end{remark}

\begin{proof}\textit{(Proof of Theorem~\ref{thm.4.6})}
Part (a) is Proposition \ref{prop:Moodyext} (a).

\noindent
For part (b) abbreviate $\omegaG:=\nuWG(\hx)$ and recall that Mirsky 2-genericity of $\omegaG$ along $-\cA$ can be characterized as
\begin{equation*}
\lim_{n\to\infty} \frac{1}{m_G(A_n)} \int_{-A_n} \phi_{c_1}(S_g \omegaG)\cdot \phi_{\overline{c_2}} (S_g \omegaG)\, {\rm d}m_G(g)= \QMG(\phi_{c_1}\cdot \phi_{\overline{c_2}})
\end{equation*}
for all $c_1,c_2\in C_c(G)$, while the existence of an autocorrelation $\gamma_{\omegaG}$ along $\cA$ satisfying $(c_1*\widetilde{c_2}*\gamma_{\omegaG})(0)= \QMG(\phi_{c_1}\cdot \phi_{\overline{c_2}})$
for all $c_1,c_2\in C_c(G)$
is equivalent to
\begin{equation*}
\lim_{n\to\infty} \frac{1}{m_G(A_n)}
\left(\omegaG|_{A_n} * \widetilde{\omegaG|_{A_n}}\right)(c_1*\widetilde{c_2})
=\QMG(\phi_{c_1}\cdot \phi_{\overline{c_2}}) \quad(c_1,c_2\in C_c(G))\ .
\end{equation*}
The equality of these two limits, provided one of them exists, is shown in the proof of Theorem 5 (a) in \cite{BaakeLenz2004}. (Note that the latter equation appears in that paper on the last line of p.~1881.)
The identity $\QMG(\phi_{c_1}\phi_{\overline{c_2}})=
\dL\cdot(c_1*\widetilde{c_2}*\nufG{c_W}(\tilde 0))(0)$ can be checked
similarly to the calculation for two-point patterns in the proof of Remark 3.12 in \cite{KR2015}, compare \cite[Prop.~3]{Moody2002}.

\noindent Part (c) follows from (b) for $\chi \in \piGhat(\LL^\circ)$ instead of $\chi\in  \piGhat({\LL^\circ}')$, e.g.~by the Poisson Summation Formula as in \cite[Thm.~4.10]{RS17}.  Also note that, since $c_W$ is $H_W^{Haar}$-periodic, $\reallywidecheck{c_W}$ is supported inside $(H_W^{Haar})^\circ$, see e.g. \cite[Prop.~6.4]{BF}. Therefore, $\reallywidecheck{c_W}(\eta)=0$ if $\chi\in  \piGhat(\LL^\circ\setminus {\LL^\circ}')$.

\noindent Part (d) is a consequence of Proposition \ref{prop:Moodyext} (b). Indeed, for each $(\chi, \eta) \in \LL^\circ$, consider $y=x+\ell\in x+\LL$ and note
\begin{displaymath}
\eta(y_H)=\eta(x_H+\ell_H)=\chi(x_G) \overline{\chi(x_G)} \eta(x_H)\eta(\ell_H)=\chi(x_G)\eta(x_H)\overline{\chi(y_G)} \ ,
\end{displaymath}
where we used $\chi(\ell_G)\eta(\ell_H)=1$. Then
\begin{displaymath}
\begin{split}
&\frac{((\eta\circ \piH)\cdot \nuW(\hx))(A_n\times H)}{m_G(A_n)}
 = \frac{ \sum_{y \in (x+\mathcal L)\cap (A_n \times W)} \eta(y_H) }{m_G(A_n)} \\
&\hspace*{5mm}  =\chi(x_G)\eta(x_H) \frac{ \sum_{y \in (x+\mathcal L)\cap (A_n \times W)} \overline{\chi(y_G) }}{m_G(A_n)}
=
\frac{\chi(x_G)\eta(x_H)}{m_G(A_n)} \int_{A_n} \overline{\chi(t)} \, {\rm d}\nuWG(t) \ .
\end{split}
\end{displaymath}
Now, if $\nuW(\hx)$ is Mirsky 1-generic along $-\cA$,  Proposition~\ref{prop:Moodyext} (b) applied to any function $\psi \in C_c(H)$ that agrees with $\eta$ on $W$ gives \eqref{eq:FB-coeff}.
Conversely, assume that \eqref{eq:FB-coeff} holds for all $(\chi, \eta) \in \LL^\circ$. Then \eqref{eq:Moody} holds for all $\psi \in C_c(H)$ which agree on $W$ with some $\eta \in \pi^{\scriptscriptstyle  \widehat{H}}(\LL^\circ)$, and hence for all linear combinations of such functions. The density of $\pi^{\scriptscriptstyle  \widehat{H}}(\LL^\circ)$ in $\widehat{H}$ implies that the set
\begin{align*}
A&:= \{ \psi \in C_c(H) : \exists n \in \mathbb N, c_1, \ldots, c_n \in \mathbb C, (\chi_1,\eta_1), \ldots, (\chi_n,\eta_n) \in \LL^0 \\
&\mbox{ such that } \psi(h)= \sum_{k=1}^n c_k \eta_k(h) \text{ for all } h \in W \}
\end{align*}
is an algebra separating the points and hence is dense in $C_0(H)$. This immediately implies that \eqref{eq:Moody} holds for all $\eta \in C_c(H)$, giving Mirsky 1-generiticity for
 $\nuW(\hx)$.

\noindent Part (e) follows from (b)-(d) and Fact~\ref{fact:bap}.
\end{proof}

\section{A class of examples}\label{sec:examples}

This section focusses on cut-and-project schemes $(G,H,\LL)$ with relatively compact Borel window $W'=W\setminus V$, where $V,W\subseteq H$ are compact sets satisfying $V\subseteq W$. Within that setting, one may construct configurations that illustrate the statements of Theorem~\ref{theo:B2gen} and Theorem~\ref{thm.4.6}, without having a window being compact modulo $0$ or without being of extremal density.

\subsection{Results for the general setting}\label{genset}

In order to apply Theorem~\ref{theo:B2gen}, one needs to determine the Haar periods of $W'$. In that context, the following notion appears to be relevant.

\begin{definition}[Haar thinness]
Let $H$ be an LCA group with Haar measure $m_H$. Consider Borel sets $V\subseteq W\subseteq H$. We say that $V$ is \textit{Haar thin} in $W$ if for all open $U\subseteq H$ such that $m_H(U\cap V)>0$ we have $m_H(U\cap V)<m_H(U\cap W)$.
\end{definition}

\begin{lemma}\label{lemma:Hthin-cm0}
If $V$ is Haar thin in $W$ and $m_H(V)>0$, then $W\setminus V$ is not compact modulo $0$.
\end{lemma}

\begin{proof}
Suppose for a contradiction that $W\setminus V=K$ modulo $0$ for some compact $K\subseteq H$. For any open $U\subseteq H$, by Haar thinness
$m_H(U\cap K)=m_H(U\cap W\setminus V)=0$ if and only if $m_H(U\cap W)=0$.
As $H$ is second countable, this implies
$K=W$ modulo $0$, i.e. $m_H(V)=0$ in contradiction to the assumption $m_H(V)>0$.
\end{proof}

\begin{lemma}[Haar periods]\label{lem:Hthin}
Let $H$ be an LCA group with Haar measure $m_H$. Let $W\subseteq H$ be compact and assume that the Borel set $V\subseteq W$ is Haar thin in $W$. Then $W'=W\setminus V$ satisfies
\begin{displaymath}
H_{W'}^{Haar}=H_{W}^{Haar}\cap H_{V}^{Haar} \ .
\end{displaymath}
\end{lemma}
\begin{proof}
Recall $h\in H_W^{Haar}$ if and only if $m_H((W+h)\setminus W)=0$. The inclusion $H_{W}^{Haar}\cap H_{V}^{Haar}\subseteq H_{W'}^{Haar}$ can be inferred from the standard estimate $(W'+h)\setminus W'\subseteq ((W+h)\setminus W) \cup (V\setminus (V+h))$. For the reverse inclusion fix arbitrary $h\in H_{W'}^{Haar}$ and note
\begin{displaymath}
\begin{split}
m_H((W+h)\setminus W)&= m_H((W'+h)\setminus W)+ m_H((V+h)\setminus W)= m_H((V+h)\setminus V)\ ,
\end{split}
\end{displaymath}
where we used $m_H((V+h)\cap W')=m_H((V+h)\cap (W'+h))=0$ in the second equation.
To conclude the argument, note first $0=m_H(W'\setminus W)=m_H((W'+h)\setminus W)$. As $V$ is Haar thin in the compact set $W$, by shift invariance of $m_H$ this implies $0=m_H((V+h)\setminus W)=m_H((W+h)\setminus W)$. Hence $h \in H_W^{Haar}\cap H_V^{Haar}$.
\end{proof}

One may now consider examples $\nufG{W'}(\hx)$ constructed from maximal density configurations  $\nufG{W}(\hx)$ and  $\nufG{V}(\hx)$. As their diffraction can be explicitly computed in particular examples such as $k$-free integers, see e.g.~the references given in \cite[Sec.~5]{HuckRichard15}, these may serve to illustrate the statements in Theorem~\ref{thm.4.6}.

\begin{lemma}[diffraction]\label{lemma:generic}
Let $(G,H,\LL)$ be a cut-and-project scheme with two compact windows $V\subseteq W\subseteq H$. Assume that, for given $\hx\in \hX$, both $\nufG{V}(\hx)$ and $\nufG{W}(\hx)$ have maximal density along the same averaging sequence, i.e., there exists  a van Hove sequence $\cA=(A_n)_n$ in $G$  such that
\begin{displaymath}
\begin{split}
\lim_{n\to\infty} \frac{1}{m_G(A_n)}\nufG{V}(\hx)(A_n)&=\dL\cdot m_H(V) \ ,\\
\lim_{n\to\infty} \frac{1}{m_G(A_n)}\nuWG(\hx)(A_n)&=\dL\cdot m_H(W) \ .
\end{split}
\end{displaymath}
Consider $W'=W\setminus V$. Then the following hold.
\begin{itemize}
  \item[(a)]  $\nufG{W'}(\hx)$ is Besicovitch almost periodic along $\cA$.
  \item[(b)] $\nufG{W'}(\hx)$ has Fourier--Bohr coefficients along $\cA$ given by 
  \[
a_\chi^\cA(\nufG{W'}(\hx))=
\left\{\begin{array}{cc}
\dL \cdot \overline{\chi(x_G)} \cdot  \overline{\eta(x_H)} \cdot \reallywidecheck{1_{W'}}(\eta)  & \mbox{ if } \chi \in \piGhat({\LL^\circ})\\
 0  & \mbox{ otherwise }
\end{array} \right. \ .
\]
  \item[(c)] $\nufG{W'}(\hx)$ has autocorrelation and  diffraction along $\cA$ given by 
\[
\gamma_{\nufG{W'}(\hx)}=\dL \cdot \nufG{c_{W'}}(\tilde 0) \ , \qquad \widehat{\gamma_{\nufG{W'}(\hx)}}= \dL^2 \cdot  \sum_{\chi \in \piGhat({\LL^\circ}')} \reallywidecheck{c_{W'}}(\eta)  \cdot \delta_\chi \ .
\]

\end{itemize}

\end{lemma}

\begin{remark}[Mirsky genericity]\label{rem:approx}
In conjunction with Theorem~\ref{thm.4.6}, the previous result implies that $\nu_{W'}(\hx)$ is  Mirsky 1-generic along $-\cA$, and that $\nufG{W'}(\hx)$ is Mirsky 2-generic along $-\cA$. Using approximation by regular model sets as in \cite{BHS, S20}, one may in fact show that $\nufG{W'}(\hx)$ is Mirsky generic along $-\cA$, without resorting to Besicovitch almost periodicity.
\end{remark}

\begin{proof} \textit{(Proof of Lemma~\ref{lemma:generic})}
By \cite[Prop.~3.39]{LSS2020a}, both $\nufG{W}(\hx)$ and $\nufG{V}(\hx)$ are Besicovitch almost periodic. Hence $\nufG{W'}(\hx)=\nufG{W}(\hx)-\nufG{V}(\hx)$ is Besicovitch almost periodic, compare \cite[Prop.~3.8]{LSS2020a}. This proves (a).

\noindent As to part (b), note that we have 
\[
a_\chi^\cA(\nufG{W'}(\hx))= a_\chi^\cA(\nufG{W}(\hx))-a_\chi^\cA(\nufG{V}(\hx)) \ .
\]
Hence (b) follows from \cite[Cor.~3.40]{LSS2020a} applied to $\nufG{W}(\hx)$ and $\nufG{V}(\hx)$.

\noindent The proof of part (c) only uses the validity of parts (a) and (b) but not the maximal density assumptions of the lemma:
The diffraction formula 
\[
\widehat{\gamma_{\nufG{W'}(\hx)}}= \dL^2 \cdot  \sum_{\chi \in \piGhat({\LL^\circ})} \reallywidecheck{c_{W'}}(\eta)  \cdot \delta_\chi
\]
follows from (a), (b) and Fact~\ref{fact:bap}. On the other hand, Lemma~3.6 and Theorem~4.10 from~\cite{RS17} give that $\gamma=\dL \cdot \nufG{c_{W'}}(\tilde 0)$ is Fourier transformable and
\[
\widehat{\gamma} =  \dL^2 \cdot  \sum_{\chi \in \piGhat({\LL^\circ})} \reallywidecheck{c_{W'}}(\eta)  \cdot \delta_\chi   \ .
\]
We thus get $\gamma=\gamma_{\nufG{W'}(\hx)}$ from double Fourier transformability \cite[Thm.~4.12]{RS17}. Since $c_{W'}$ is $H_{W'}^{Haar}$-periodic, we can restrict the summation over  $\piGhat({\LL^\circ})$ to  $\piGhat({\LL^\circ}')$. This proves (c). 
\end{proof}

\subsection{An example from $\BB$-free sets}

Assume that $H$ is compact. Then any weak model set is a subset of the lattice $\Lambda_H$. The trivial choice $W=H$ leads to comparing a weak model set $\Lambda_V$ to its lattice complement $\Lambda_{H\setminus V}$. This applies to so-called sets of multiples, which are usually studied dynamically through their complementary  $\BB$-free sets, see \cite{BKKL2015} and  \cite{Ka-Ke-Le}.

For an example beyond extremal density, let us consider the set of cube-free integers that are not square-free. An appropriate cut-and-project scheme $(\Z, H, \mathcal{L})$ has compact internal space $H= \prod_{p \in \mathcal{P}} \Z/(p^3\Z)$, with $\mathcal P$ denoting the set of all primes. Moreover $\mathcal{L}= \{ (n, \Delta(n) ) : n \in \Z \}$, where $\Delta : \Z \to H$ denotes the natural embedding $ \Delta(n)=(n,n,n,\ldots)\in H$. Consider the compact sets  $V \subseteq W\subseteq H$  given by
\begin{displaymath}
W= \prod_{p \in \mathcal{P}}\left( \Z/(p^3\Z) \setminus \{ 0 \} \right) \ , \qquad
V= \prod_{p \in \mathcal{P}}\left( \Z/(p^3\Z) \setminus \{ 0, p^2, 2 p^2, \ldots, (p-1)p^2 \} \right) \ .
\end{displaymath}
Then $\Lambda_{V}$ and $\Lambda_{W}$ are the sets of square-free integers and cube-free integers, respectively, and $\Lambda_{W'}$ is the set of integers that are cube-free but not square-free.

\smallskip

Note  that $W'$ is Haar aperiodic and fails to be compact modulo $0$. Together with  Lemma~\ref{lem:Hthin} and Lemma~\ref{lemma:Hthin-cm0}, this is an immediate consequence of the following result. Recall that $W$ is Haar regular if $U\cap W\ne\varnothing$ implies $m_H(U\cap W)>0$ for any open $U\subseteq H$, see \cite[Def.~3.10]{KR19}.

\begin{lemma}\label{lem:ht}
Both $V$ and $W$ are Haar regular, and $V$ is Haar thin in $W$. Moreover $W$ is Haar aperiodic.
\end{lemma}

\begin{proof}
Both Haar regularity and Haar thinness can be checked by restricting to open cylinder sets $U_S(h)\subseteq H$ as defined in \cite{Ka-Ke-Le}, for $h\in H$ and finite $S\subset\{p^3:p\in\cP\}$. But for those cylinder sets the claims are obvious due to the product structure of $m_H$. For Haar regularity of $W$, note $m_H(W)>0$ by positive density of cube-free integers. Thus $m_H(U_S(h)\cap W)>0$ for $h\in W$, as intersecting by $U_S(h)$ affects only finite many coordinates. An analogous argument shows Haar regularity of $V$. A similar argument also shows Haar thinness, noting that $m_H(W)>0$ implies $m_H(W')>0$. As $W$ clearly is aperiodic, Haar aperiodicity follows from Haar regularity by \cite[Rem.~3.12]{KR19}.
\end{proof}

Note further that both $\Lambda_V$ and $\Lambda_W$ are weak model sets of maximal density with respect to $A_n=[-n,n]$, see e.g.~\cite[Sec.~5.2]{HuckRichard15} and references therein.  Moreover the window $W'$ satisfies $(W')^\circ=\varnothing$ and $\overline{W'}=W$. For the latter claim, note that due to $\overline{W'}\subseteq W=V\cup \overline{W'}$ it suffices to show $V\subseteq \overline{W'}$. But this is obvious as $V$ is Haar regular and Haar thin in $W$.

\smallskip

To summarise, the example $\Lambda_{W'}$ of cube-free integers that are not square-free has a window $W'$ that is not compact modulo $0$. As  $W'$ is Haar aperiodic, Theorem~\ref{theo:B1gen} applies. Thus the dynamical spectrum of the Mirsky measure $\QMGprime$ equals $\piGhat(\LL^\circ)$. It thus coincides with the dynamical spectrum of the Mirsky measure $\QMG$ of cube-free integers. Note that the dynamical spectrum can be identified with the discrete group $\widehat H$. Whereas $\Lambda_{W'}$ fails to have extremal density along $\cA$, both Theorem~\ref{thm.4.6} and Remark~\ref{rem:conclusions} apply to  $\Lambda_{W'}$, due to Lemma~\ref{lemma:generic} and Remark~\ref{rem:approx}.

\section*{Acknowledgements}
NS was supported by the Natural Sciences and Engineering Council of Canada (NSERC) via grant 2020-00038, and he is grateful for the support.

\end{document}